\documentclass[a4paper,12pt,reqno]{amsart}

\usepackage{amsmath}
\usepackage{amscd}
\usepackage{amssymb}
\usepackage{mathrsfs}
\usepackage[left=2.5cm,right=2.5cm,bottom=3cm,top=3cm]{geometry}
\newtheorem{thm}{Theorem}[section]
\newtheorem{cor}[thm]{Corollary}

\newtheorem{lem}[thm]{Lemma}

\newtheorem{prop}[thm]{Proposition}
\theoremstyle{definition}

\theoremstyle{remark}
\newtheorem{rem}[thm]{Remark}

\numberwithin{equation}{section}
\begin{document}

\title{The continuity method on minimal elliptic K\"ahler surfaces}

\author{Yashan Zhang}
\address{Department of mathematics, University of Macau, Macau, China}
\email{yashanzh@163.com}
\thanks{Y. Zhang is partially supported by the Science and Technology Development Fund (Macao S.A.R.) Grant FDCT/ 016/2013/A1 and the Project MYRG2015-00235-FST of the University of Macau}

\author{Zhenlei Zhang}
\address{Department of mathematics, Capital Normal University, Beijing, China}
\email{zhleigo@aliyun.com}
\thanks{Z. Zhang is partially supported by NSFC 11431009}

\begin{abstract}
We prove that, on a minimal elliptic K\"ahler surface of Kodaira dimension one, the continuity method introduced by La Nave and Tian in \cite{LT} starting from any initial K\"{a}hler metric converges in Gromov-Hausdorff topology to the metric completion of the generalized K\"{a}hler-Einstein metric on its canonical model constructed by Song and Tian in \cite{ST06}.
\end{abstract}

\maketitle

\section{Introduction}
In \cite{LT}, La Nave and Tian introduced a new approach to the Analytic Minimal Model Program. It is a continuity method of complex Monge-Amp\`ere equations. In this note, we will study the geometric convergence of this continuity equation on minimal elliptic surfaces.

\par Let $X$ be a minimal elliptic K\"ahler surface of Kodaira dimension $kod(X)=1$. By definition (see, e.g., \cite[Section I.3]{M} or \cite[Section 2.2]{ST06}), there exists a holomorphic map $f:X\rightarrow\Sigma$, determined by the pluricanonical system $|mK_{X}|$ for sufficiently large integer $m$, from $X$ onto a smooth projective curve $\Sigma$ (i.e., the canonical model of $X$), such that the general fiber is a smooth elliptic curve and all fibers are free of $(-1)$-curves. Set $\Sigma_{reg}:=\{s\in \Sigma|X_{s}:=f^{-1}(s)$ is a nonsingular fiber$\}$ and $X_{reg}=f^{-1}(\Sigma_{reg})$. Assume  $\Sigma\setminus\Sigma_{reg}=\{s_{1},\ldots,s_{k}\}$ and let $m_{i}F_{i}=X_{s_{i}}$ be the corresponding singular fiber of multiplicity $m_{i}$, $i=1,\ldots,k$. We refer readers to \cite[Section I.5]{M} for several interesting examples of minimal elliptic surfaces.
\par In \cite{ST06}, Song and Tian proved that there exists a unique generalized K\"{a}hler-Einstein current $\chi_{\infty}$ on $\Sigma$, i.e., $\chi_{\infty}$ is a closed positive $(1,1)$-current on $\Sigma$ such that $\chi_{\infty}$ is smooth on $\Sigma_{reg}$ and $Ric(\chi_{\infty})=-\sqrt{-1}\partial\bar{\partial}\log\chi_{\infty}$ is a well-defined $(1,1)$-current on $\Sigma$ satisfying
\begin{equation}\label{KaEi}
Ric(\chi_{\infty})=-\chi_{\infty}+\omega_{WP}+2\pi\sum_{i=1}^{k}\frac{m_{i}-1}{m_{i}}[s_{i}],
\end{equation}
where $\omega_{WP}$ is the induced Weil-Petersson form and $[s_{i}]$ is the current of integration associated to the divisor $s_{i}$ on $\Sigma$.
\par In this paper we consider the following continuity method introduced by La Nave and Tian in \cite{LT} and Rubinstein in \cite{Ru} starting from any initial K\"ahler metric $\omega_{0}$ on $X$,
\begin{equation}\label{lt}
\left\{
\begin{aligned}
(1+t)\omega(t)&=\omega_{0}-tRic(\omega(t))\\
\omega(0)&=\omega_{0}.
\end{aligned}
\right.
\end{equation}
According to \cite{LT}, a solution $\omega=\omega(t)$ to (\ref{lt}) exists uniquely for all $t\geq0$ and the Ricci curvature of $\omega(t)$ satisfies
\begin{equation}\label{Ric bound}
Ric(\omega(t))\geq-2\omega(t)
\end{equation}
for all $t\geq1$.
\par Our main result is the following
\begin{thm}\label{main}
Assume as above, we have
\begin{itemize}
\item[(1)] As $t\to\infty$, $\omega(t)\to f^{*}\chi_{\infty}$ as currents on $X$ and, for any given compact subset $V$ of $X_{reg}$, there exists a constant $\alpha_{V}\in(0,1)$ such that $\omega(t)\to f^{*}\chi_{\infty}$ in $C^{1,\alpha_{V}}(V,\omega_0)$-topology;
\item[(2)] For any $s\in\Sigma_{reg}$, $(1+t)\omega(t)|_{X_s}$ converges in $C^\infty(X_s,\omega_0|_{X_s})$-topology to the unique flat metric in class $[\omega_0|_{X_s}]$ as $t\to\infty$.
\end{itemize}
Moreover, if we let $(X_{\infty},d_{\infty})$ be the metric completion of $(\Sigma_{reg},\chi_{\infty})$, then
\begin{itemize}
\item[(3)] $(X_{\infty},d_{\infty})$ is a compact length metric space and $X_{\infty}$ is homeomorphic to $\Sigma$ as a projective variety;
\item[(4)] As $t\to\infty$, $(X,\omega(t))\to(X_{\infty},d_{\infty})$ in Gromov-Hausdorff topology.
\end{itemize}
\end{thm}
We would like to explain how our result fits the existing literatures. According to the Analytic Minimal Model Program proposed in \cite{ST06,ST12,ST16,LT}, the K\"ahler-Ricci flow and the continuity method should deform any K\"ahler metric on a smooth minimal model (i.e., a K\"ahler manifold with nef canonical line bundle), say $M$, to a canonical metric or its metric completion on its canonical model $M_{can}$ in Gromov-Hausdorff topology. If in addition we assume that the canonical line bundle $K_M$ of $M$ is semi-ample, then by semi-ample fibration theorem (see \cite{La,Ue}) there exists a fiber space map determined by the pluricanonical system of $K_{M}$:
\begin{equation}\label{fibration}
\pi:M\to M_{can}.
\end{equation}
\par If $dim(M_{can})=0$, then $M$ is a Calabi-Yau manifold. By a classical result of Cao \cite{C}, K\"ahler-Ricci flow will deform any K\"ahler metric to the unique Ricci-flat K\"ahler metric in the same K\"ahler class smoothly. It can be easily checked that the same result holds for the continuity method.
\par If $dim(M_{can})=dim(M)$, i.e, $M$ is a smooth minimal model of general type, the expected geometric convergence is obtained for K\"ahler-Ricci flow when $dim(M)\le3$ (see \cite{GSW,TZe}). For the continuity method, the convergence is obtained for any dimension in \cite{LTZ}.
\par The remaining case is $0<dim(M_{can})<dim(M)$. In this case, the geometric convergence of K\"ahler-Ricci flow is obtained in \cite{TWY} (see \cite{FZ,Gi} for certain special cases) when $M_{can}$ is smooth and the semi-ample fibration \eqref{fibration} does not admit any singular fiber. In general, this problem is largely open. Our Theorem \ref{main} confirms this expected picture for the continuity method when $dim(M)=2$. In fact, our argument can also apply to K\"ahler-Ricci flow on minimal elliptic K\"ahler surface of Kodaira dimension one if all its singular fibers are of type $mI_0$ (see Remark \ref{r3.1}).\\

The rest of this paper is organized as follows. We will prove parts (1) and (2) of Theorem \ref{main} in Section 2 and prove parts (3) and (4) in Section \ref{geomconv}. The key observation in the proof is that the limit space $(X_\infty,d_\infty)$ is compact.

\section{Estimates and local convergence}\label{est}
In this section we will derive necessary estimates and prove parts (1) and (2) of Theorem \ref{main}.
\par As the first step, we will reduce (\ref{lt}) to a scalar equation of Monge-Amp\`{e}re type as in \cite{LT,LTZ}. Let $\chi$ be the restriction on $\Sigma$ of a multiple of the Fubini-Study metric of a projective space and $\Omega$ a smooth positive volume form on $X$ with $\sqrt{-1}\partial\bar{\partial}\log\Omega=f^{*}\chi$. Set $\omega_{t}:=\frac{1}{1+t}\omega_{0}+\frac{t}{1+t}f^{*}\chi$. Then (\ref{lt}) can be reduced to the following equation of $\varphi=\varphi(t)$

\begin{equation}\label{lt1}
\left\{
\begin{aligned}
(\omega_{t}+\sqrt{-1}\partial\bar{\partial}\varphi)^{2}&=(1+t)^{-1}e^{\frac{1+t}{t}\varphi}\Omega\\
\varphi(0)&=0.
\end{aligned}
\right.
\end{equation}
Namely, $\omega(t)=\omega_{t}+\sqrt{-1}\partial\bar{\partial}\varphi(t)$ solves (\ref{lt}) if $\varphi(t)$ solves (\ref{lt1}). We remark that the factor $(1+t)^{-1}$ in the right hand side of \eqref{lt1} comes from the cohomology class and formal scale of volume, which in particular is crucial in obtaining a uniform bound of K\"ahler potential $\varphi(t)$, see Lemma \ref{C0'}.

Next, following \cite{To10}, we fix a smooth nonnegative function $\sigma$ on $X$, which vanishes exactly on singular fibers and satisfies
\begin{equation}\label{00}
\sigma\leq1, \sqrt{-1}\partial\sigma\wedge\bar{\partial}\sigma\leq C f^{*}\chi, -C f^{*}\chi\leq\sqrt{-1}\partial\bar{\partial}\sigma\leq C f^{*}\chi
\end{equation}
on $X$ for some constant $C$.

\begin{lem}\label{C0'}
There exists a constant $C>0$ such that for any $t\geq1$,
\begin{equation}\label{C0}
\|\varphi(t)\|_{C^{0}(X)}\leq C,
\end{equation}
or equivalently,
\begin{equation}\label{vol}
\frac{1}{C(1+t)}\Omega\leq\omega(t)^{2}\leq\frac{C}{1+t}\Omega.
\end{equation}
\end{lem}
\begin{proof}
It suffices to prove (\ref{C0}). Firstly note that $\omega_{t}^{2}\leq C\frac{\Omega}{(1+t)}$. By applying the maximum principle in (\ref{lt1}), the upper bound of $\varphi$ follows easily. Then (\ref{C0}) can be proved by the Moser iteration as in \cite{ST06,Y}. We remark that this lemma also follows from general theory of degenerate complex Monge-Amp\`{e}re equations, see \cite{DP,EGZ08}.
\end{proof}

\begin{lem}\label{C2a'}
There exists a positive constant $C$ such that for all $t\geq1$
\begin{equation}\label{C2a}
tr_{\omega}f^{*}\chi\leq C.
\end{equation}
\end{lem}

\begin{proof}
Along (\ref{lt}), the Schwarz Lemma argument (see, e.g., \cite{Y78,ST06}) gives
\begin{equation}
\Delta_{\omega}(\log tr_{\omega}f^{*}\chi-2(A+1)\varphi)\geq tr_{\omega}f^{*}\chi\nonumber-4(A+1)-2\nonumber
\end{equation}
for some fixed large constant $A$. Using the maximum principle and the $C^{0}$-estimate (\ref{C0}), we have proved (\ref{C2a}).
\end{proof}

Using (\ref{C2a}) and the same argument in \cite{ST06}, we can find a positive constant $\lambda_{1}$ such that for all $t\geq1$,
\begin{equation}\label{2.44}
\sup_{X_{s}}\varphi-\inf_{X_{s}}\varphi\leq\frac{C (1+t)^{-1}}{\sigma^{\lambda_{1}}}.
\end{equation}

Next, as in \cite{ST06}, we define a function $\bar{\varphi}$ on $\Sigma$ by
\begin{equation}
\bar{\varphi}(s)=\frac{\int_{X_{s}}\varphi(\omega_{0}|_{X_{s}})}{\int_{X_{s}}\omega_{0}|_{X_{s}}}\nonumber.
\end{equation}
Then (\ref{2.44}) implies that
\begin{equation}\label{2.5}
(1+t)|\varphi-\bar{\varphi}|\leq \frac{C}{\sigma^{\lambda_{1}}}.
\end{equation}
\begin{lem}\label{lem2}
\begin{equation}
\Delta_{\omega}((1+t)(\varphi-\bar{\varphi}))\leq-tr_{\omega}\omega_{0}+\frac{1}{Vol(X_{s})}tr_{\omega}\left(\int_{X_{s}}\omega_{0}^{2}\right)+2(1+t).
\end{equation}
\end{lem}
\begin{proof}
The proof is the same as \cite[Lemma 5.9]{ST06}.
\end{proof}

Next lemma can be easily checked by a direct computation.
\begin{lem}
There exists a positive constant $C$ such that for all $t\geq1$,
\begin{equation}
\Delta_{\omega}\log tr_{\omega}((1+t)^{-1}\omega_{0})\geq-Ctr_{\omega}\omega_{0}-C.
\end{equation}
\end{lem}

\begin{lem}\label{C2b'}
There exist positive constants $\lambda_{2}$ and $C$ such that for all $t\geq1$,
\begin{equation}\label{C2b}
tr_{\omega}((1+t)^{-1}\omega_{0})\leq Ce^{C\sigma^{-\lambda_{2}}}.
\end{equation}
\end{lem}

\begin{proof}
Set $H=\sigma^{\lambda_{2}}(\log tr_{\omega}((1+t)^{-1}\omega_{0})-A(1+t)(\varphi-\bar{\varphi}))$. By the Schwarz Lemma argument (see, e.g., \cite{Y78,ST06}), we have
\begin{equation}
\Delta_{\omega}H\geq\frac{A}{3}\sigma^{\lambda_{2}}tr_{\omega}\omega_{0}+2Re\left(\frac{\nabla H\overline{\nabla}\sigma^{\lambda_{2}}}{\sigma^{\lambda_{2}}}\right)-3A(1+t)\nonumber
\end{equation}
if we choose $\lambda_{2}$ and $A$ large enough. Now by the maximum principle and (\ref{2.5}), (\ref{C2b}) follows.
\end{proof}

\begin{cor}
There exist positive constants $C$ and $\lambda_{2}$ such that for all $t\geq1$,
\begin{equation}\label{ric}
Ric(\omega)\leq Ce^{C\sigma^{-\lambda_{2}}}\omega.
\end{equation}
\end{cor}

\begin{proof}
Along the continuity equation (\ref{lt}) we have, for $t\geq1$,
\begin{equation}
Ric(\omega)=t^{-1}\omega_{0}-\frac{(1+t)}{t}\omega\nonumber.
\end{equation}
Thus (\ref{ric}) follows from Lemma \ref{C2b'} immediately.
\end{proof}

Combining (\ref{C2a}) and (\ref{C2b}) we have
\begin{equation}
tr_{\omega}\omega_{t}\leq Ce^{C\sigma^{-\lambda_{2}}}\nonumber.
\end{equation}
Then
\begin{align}
tr_{\omega_{t}}\omega&\leq (tr_{\omega}\omega_{t})\frac{\omega^{2}}{\omega_{t}^{2}}\nonumber\\
&\leq C(tr_{\omega}\omega_{t})\frac{(1+t)^{-1}\Omega}{(1+t)^{-1}\omega_{0}\wedge f^{*}\chi}\nonumber\\
&\leq Ce^{C\sigma^{-\lambda_{2}}}\sigma^{-\lambda_{2}'}\nonumber\\
&\leq Ce^{C\sigma^{-\lambda_{3}}}\nonumber
\end{align}
In conclusion,
\begin{equation}\label{C2c}
C^{-1}e^{-C\sigma^{-\lambda_{2}}}\omega_{t}\leq\omega\leq Ce^{C\sigma^{-\lambda_{3}}}\omega_{t}.
\end{equation}

Let $t_{j}$ be any time sequence converging to $\infty$. Since the cohomology class of $\omega(t)$ is bounded and $\varphi(t)$ is uniformly bounded for all $t\ge1$, using the weak compactness of currents, we may assume that $\omega(t_{j})$ converges to a limit closed positive $(1,1)$-current $\chi_{\infty}$ (see, e.g., \cite{De}), which a priori depends on the given sequence. Note that  $\chi_{\infty}\in[f^{*}\chi]$, $\chi_{\infty}=f^{*}\chi+\sqrt{-1}\partial\bar{\partial}\varphi_{\infty}$ and $\varphi(t_{j})\rightarrow\varphi_{\infty}$ in $L^{1}(X,\omega_{0}^{2})$, which in particular implies that $\varphi_{\infty}\in L^{\infty}(X)$. Indeed, after passing to a subsequence, we may assume $\varphi(t_j)\to\varphi_\infty$ a.e. on $X$ and hence, by Lemma \ref{C0'}, $|\varphi_\infty|\le C$ a.e. on $X$. Then, for any $x\in X$ and a fixed local potential $u$ of $f^*\chi$ (i.e. $f^*\chi=\sqrt{-1}\partial\bar\partial u$) on $B_{\omega_0}(x,\rho_0)$, applying e.g. \cite[Theorem K.15]{Gu} gives
$$|(u+\varphi_\infty)(x)|=\left|\lim_{\epsilon\to0}\frac{\int_{B_{\omega_0}(x,\epsilon)}(u+\varphi_\infty)\omega_0^2}{\int_{B_{\omega_0}(x,\epsilon)}\omega_0^2}\right|\le C$$
and so
$$|\varphi_\infty(x)|\le C$$
for some uniform constant $C$.

Moreover, using the estimates we have obtained above, we can assume that $\varphi(t_{j})\rightarrow\varphi_{\infty}$ in $C_{loc}^{1,\alpha}(X_{reg},\omega_{0})$ for any $\alpha\in(0,1)$.
\begin{lem}\label{lemlem}
There exists a function $\tilde\varphi_\infty\in PSH(\Sigma,\chi)\cap L^\infty(\Sigma)$ such that $f^*\tilde\varphi_\infty=\varphi_\infty$.
\end{lem}
\begin{proof}
We present a proof similar to \cite[Theorem 6.3]{EGZ}. It suffices to show that $\varphi_\infty$ is constant on every fiber $X_s=f^{-1}(s)$. For $s\in \Sigma_{reg}$, i.e. $X_s$ is a nonsingular fiber,
when restricting $\chi_{\infty}$ to such $X_{s}$, we see that $\sqrt{-1}\partial\bar{\partial}\varphi_{\infty}|_{X_{s}}\geq0$. Hence $\varphi_{\infty}$ is constant on every nonsingular fiber $X_s$. For a singular fiber, $X_0$, by Hironaka's theorems (see \cite{Hi}) we fix an embedded resolution $\tilde f:\tilde X\to X$ of singularities of $X_0$, i.e., $\tilde X$ is a compact complex manifold, $\tilde f:\tilde X\to X$ is a holomorphic surjective map and is biholomorphic over $X\setminus X_0$, and the proper transform of $X_0$, denoted by $\tilde X_0$, is a smooth connected submanifold of $\tilde X$. Now we pullback $\chi_\infty$ to $\tilde X$ to obtain $\pi^*\chi_{\infty}=\pi^*f^*\chi+\sqrt{-1}\partial\bar\partial\pi^*\varphi_\infty$, which is a closed positive $(1,1)$-form on $\tilde X$. Now, as before, we can restrict $\pi^*\chi_\infty$ to $\tilde X_0$ to see that $\pi^*\varphi_\infty$ is plurisubharmonic on $\tilde X_0$ and hence is constant on $\tilde X_0$, which implies that $\varphi_\infty$ is constant on $X_0$.
\par Lemma \ref{lemlem} is proved.
\end{proof}
In the following, we identify $\tilde\varphi_\infty$ and $\varphi_\infty$.

On the other hand, let $\omega_{SF}:=\omega_{0}+\sqrt{-1}\partial\bar{\partial}\rho_{SF}$, where $\rho_{SF}$ is a smooth function on $X_{reg}$, be the semi-flat $(1,1)$-form defined by \cite[Lemma 3.1]{ST06}. Define $F:=\frac{\Omega}{2\omega_{SF}\wedge f^{*}\chi}$, which can be seen as a function $\in L^{1+\epsilon}(\Sigma,\chi)$ (see \cite{ST06,ST12,He}) and is smooth on $\Sigma_{reg}$. It was proved in \cite{ST06} (see also \cite{Ko1} for more general theory) that there exists a unique solution $\hat{\varphi}\in PSH(\Sigma,\chi)\cap C^{0}(\Sigma)\cap C^{\infty}(\Sigma_{reg})$ to the following equation on $\Sigma$:
\begin{equation}\label{KE}
\chi+\sqrt{-1}\partial\bar{\partial}\hat{\varphi}=Fe^{\hat{\varphi}}\chi.
\end{equation}

Here we will use the above estimates to prove the existence of a bounded solution of (\ref{KE}). Precisely, we have
\begin{lem}\label{KE'}
$\varphi_{\infty}$ is a bounded solution of equation (\ref{KE}).
\end{lem}

\begin{proof}
Since $\varphi_{\infty}$ is bounded, $\chi+\sqrt{-1}\partial\bar\partial\varphi_\infty$ takes no mass on pluripolar sets, e.g. $\Sigma\setminus\Sigma_{reg}$ (see e.g. \cite{Ko05}). Moreover, using $F\in L^{1+\epsilon}(\Sigma)$ and H\"older inequality, one easily sees that $e^{\varphi_\infty}F\chi$ also takes no mass on $\Sigma\setminus\Sigma_{reg}$. Therefore, it suffices to show that for any given $K\subset\subset\Sigma_{reg}$ and any given $\phi\in C_{0}^{\infty}(K)$,
\begin{equation}\label{KE2}
\int_{\Sigma}\phi(\chi+\sqrt{-1}\partial\bar{\partial}\varphi_{\infty})=\int_{\Sigma}\phi Fe^{\varphi_{\infty}}\chi.
\end{equation}
To this end, we use an argument similar to \cite[Theorem 4.1]{To10}. Firstly, using the equation (\ref{lt1}), we have
\begin{equation}\label{KE3}
\int_{X}(f^{*}\phi)e^{\frac{1+t}{t}\varphi}\Omega=\int_{X}(f^{*}\phi)(1+t)(\omega_{t}+\sqrt{-1}\partial\bar{\partial}\varphi)^{2}.
\end{equation}
As $t_{j}\rightarrow\infty$, the left hand side of (\ref{KE3}) will go to
\begin{equation}\label{KE4.1}
\int_{X}(f^{*}\phi)e^{\varphi_{\infty}}\Omega=2\int_{X}(f^{*}\phi)e^{\varphi_{\infty}}F\omega_{SF}\wedge f^{*}\chi=2\int_{\Sigma}\phi Fe^{\varphi_{\infty}}\chi\int_{X_{s}}\omega_{SF}|_{X_{s}}.
\end{equation}
For the right hand side of (\ref{KE3}), we have
\begin{align}
&\int_{X}(f^{*}\phi)(1+t)(\omega_{t}+\sqrt{-1}\partial\bar{\partial}\varphi)^{2}\nonumber\\
&=\int_{X}(f^{*}\phi)(1+t)\left(\frac{1}{1+t}\omega_{0}+\sqrt{-1}\partial\bar{\partial}(\varphi-\bar{\varphi})+\frac{t}{1+t}f^{*}\chi+\sqrt{-1}\partial\bar{\partial}\bar{\varphi}\right)^{2}\nonumber\\
&=\int_{X}(f^{*}\phi)\frac{1}{1+t}\omega_{0}^{2}\nonumber\\
&+2\int_{X}(f^{*}\phi)\omega_{0}\wedge\sqrt{-1}\partial\bar{\partial}(\varphi-\bar{\varphi})\nonumber\\
&+\int_{X}(f^{*}\phi)(1+t)(\sqrt{-1}\partial\bar{\partial}(\varphi-\bar{\varphi}))^{2}\nonumber\\
&+2\int_{X}(f^{*}\phi)(1+t)\sqrt{-1}\partial\bar{\partial}(\varphi-\bar{\varphi})\wedge(\frac{t}{1+t}f^{*}\chi+\sqrt{-1}\partial\bar{\partial}\bar{\varphi})\nonumber\\
&+2\int_{X}(f^{*}\phi)\omega_{0}\wedge(\frac{t}{1+t}f^{*}\chi+\sqrt{-1}\partial\bar{\partial}\bar{\varphi})\nonumber\\
&=A_{1}+A_{2}+A_{3}+A_{4}+A_{5}\nonumber.
\end{align}
(1) The first term $A_{1}$ will go to zero as $t\rightarrow\infty$.\\
(2) The second term
\begin{align}
A_{2}&=2\int_{X}(f^{*}\phi)\omega_{0}\wedge\sqrt{-1}\partial\bar{\partial}(\varphi-\bar{\varphi})\nonumber\\
&=2\int_{X}(\varphi-\bar{\varphi})\omega_{0}\wedge\sqrt{-1}\partial\bar{\partial}(f^{*}\phi)\nonumber,
\end{align}
which will go to zero as $t\rightarrow\infty$ by (\ref{2.5}).\\
(3) The third term
\begin{align}
A_{3}&=\int_{X}(f^{*}\phi)(1+t)(\sqrt{-1}\partial\bar{\partial}(\varphi-\bar{\varphi}))^{2}\nonumber\\
&=\int_{X}(1+t)(\varphi-\bar{\varphi})\sqrt{-1}\partial\bar{\partial}(f^{*}\phi)\wedge\sqrt{-1}\partial\bar{\partial}(\varphi-\bar{\varphi})\nonumber\\
&=\int_{X}(1+t)(\varphi-\bar{\varphi})\sqrt{-1}\partial\bar{\partial}(f^{*}\phi)\wedge\sqrt{-1}\partial\bar{\partial}\varphi|_{X_{s}}\nonumber.
\end{align}
Note that $(1+t)(\varphi-\bar{\varphi})$ is uniformly bounded on $K$ by (\ref{2.5}) and
\begin{equation}
-(1+t)^{-1}\omega_{0}|_{X_{s}}\leq\sqrt{-1}\partial\bar{\partial}\varphi|_{X_{s}}\leq C_{K}(1+t)^{-1}\omega_{0}|_{X_{s}}\nonumber
\end{equation}
by (\ref{C2c}). Thus $A_{3}\rightarrow0$ as $t\rightarrow\infty$.\\
(4) The fourth term
\begin{align}
A_{4}&=2\int_{X}(f^{*}\phi)(1+t)\sqrt{-1}\partial\bar{\partial}(\varphi-\bar{\varphi})\wedge(\frac{t}{1+t}f^{*}\chi+\sqrt{-1}\partial\bar{\partial}\bar{\varphi})\nonumber\\
&=2\int_{X}(1+t)(\varphi-\bar{\varphi}))\sqrt{-1}\partial\bar{\partial}(f^{*}\phi)\wedge(\frac{t}{1+t}f^{*}\chi+\sqrt{-1}\partial\bar{\partial}\bar{\varphi})\nonumber\\
&=0\nonumber,
\end{align}
since the term $\sqrt{-1}\partial\bar{\partial}(f^{*}\phi)\wedge(\frac{t}{1+t}f^{*}\chi+\sqrt{-1}\partial\bar{\partial}\bar{\varphi})=f^*(\sqrt{-1}\partial\bar{\partial}\phi\wedge(\frac{t}{1+t}\chi+\sqrt{-1}\partial\bar{\partial}\bar{\varphi}))$ and, obviously, $\sqrt{-1}\partial\bar{\partial}\phi\wedge(\frac{t}{1+t}\chi+\sqrt{-1}\partial\bar{\partial}\bar{\varphi})$ vanishes on $\Sigma$ as $dim(\Sigma)=1$.\\
(5) For the last term $A_{5}$, first note that (\ref{2.5}) implies that $\bar{\varphi}(t_{j})\rightarrow\varphi_{\infty}$ in $L^{\infty}(K)$ as $t_{j}\rightarrow\infty$. So we have
\begin{align}\label{KE4.2}
A_{5}&=2\int_{X}(f^{*}\phi)\omega_{0}\wedge(\frac{t_{j}}{1+t_{j}}f^{*}\chi+\sqrt{-1}\partial\bar{\partial}\bar{\varphi}(t_{j}))\nonumber\\
&\rightarrow2\int_{X}(f^{*}\phi)\omega_{0}\wedge(f^{*}\chi+\sqrt{-1}\partial\bar{\partial}\varphi_{\infty})\nonumber\\
&=2\int_{\Sigma}\phi (\chi+\sqrt{-1}\partial\bar{\partial}\varphi_{\infty})\int_{X_{s}}\omega_{0}|_{X_{s}}.
\end{align}
Combining (\ref{KE3}), (\ref{KE4.1}), (\ref{KE4.2}) and the fact that $\int_{X_{s}}\omega_{SF}|_{X_{s}}=\int_{X_{s}}\omega_{0}|_{X_{s}}\equiv constant$, we obtain (\ref{KE2}).
\end{proof}

Let $\chi_{\infty}=\chi+\sqrt{-1}\partial\bar{\partial}\varphi_{\infty}$. Since the bounded solution of (\ref{KE}) is unique (see e.g. \cite{ST06}), we conclude that the above convergence holds without passing to a subsequence, i.e.,
\begin{lem}\label{conv}
$\varphi(t)\rightarrow\varphi_{\infty}$ in $L^{1}(X,\omega_{0}^{2})$ and $C_{loc}^{1,\alpha}(X_{reg},\omega_{0})$ for any $\alpha\in(0,1)$ as $t\rightarrow\infty$. In particular $\omega(t)\rightarrow f^{*}\chi_{\infty}$ in the sense of currents as $t\rightarrow\infty$.
\end{lem}

\begin{rem}
Alternatively, we can easily apply arguments in \cite[Section 6]{ST06} to conclude the $L^1$-convergence in Lemma \ref{conv}.
\end{rem}

\par Next we shall prove the interior $C^{1,\alpha}$ estimate for $\omega(t)$ on $X_{reg}$, where we will apply an idea due to \cite{GTZ,HeTo} (see also \cite{FZ}). Let $B\subset\Sigma_{reg}$ small enough and $U=f^{-1}(B)$. There exists a holomorphic function $z(s)$ on $B$ such that $Im(z(s))>0$ and $U$ is biholomorphic to $B\times\mathbb{C}/\mathbb{Z}\oplus\mathbb{Z} z(s)$, which is compatible with the projection to $B$. Composing the quotient map $B\times\mathbb{C}\rightarrow B\times\mathbb{C}/\mathbb{Z}\oplus\mathbb{Z} z(s)$ with this biholomorphism, we obtain a local biholomorphism $p:B\times\mathbb{C}\rightarrow U$ such that $f\circ p(s,w)=s$ for all $(s,w)$. Moreover, by \cite{HeTo} there exists a closed semi-positive real $(1,1)$-form $\tilde{\omega}_{SF}$ on $U$ such that $\tilde{\omega}_{SF}=\omega_{0}+\sqrt{-1}\partial\bar{\partial}\rho$ for some $\rho\in C^{\infty}(U,\mathbb{R})$, $\tilde{\omega}_{SF}|_{X_{s}}$ is a flat metric on $X_{s}$ for all $s\in B$ and $p^{*}\tilde{\omega}_{SF}=\sqrt{-1}\partial\bar{\partial}\eta$, where $\eta(s,w)$ is a smooth real function on $B\times\mathbb{C}$ satisfying
\begin{equation}\label{scaling}
\eta(s,\lambda w)=\lambda^{2}\eta(s,w)
\end{equation}
for any $\lambda\in\mathbb{R}$ (note that, a priori, $\tilde\omega_{SF}$ may be different from $\omega_{SF}$ given by Lemma 3.1 of \cite{ST06}).

\par Define $\lambda_{t}:B\times\mathbb{C}\rightarrow B\times\mathbb{C}$ by $\lambda_{t}(s,w)=(s,\sqrt{1+t}\cdot w)$. Then using (\ref{scaling}) we have
\begin{align}\label{2.11}
(1+t)^{-1}\lambda_{t}^{*}p^{*}\tilde{\omega}_{SF}&=(1+t)^{-1}\lambda_{t}^{*}\sqrt{-1}\partial\bar{\partial}\eta\nonumber\\
&=(1+t)^{-1}\sqrt{-1}\partial\bar{\partial}\eta\circ\lambda_{t}\nonumber\\
&=\sqrt{-1}\partial\bar{\partial}\eta\nonumber\\
&=p^{*}\tilde{\omega}_{SF}.
\end{align}
Since $\omega_{0}$ is equivalent to $\tilde{\omega}_{SF}+f^{*}\chi$ on $U$, we know that $\omega_{t}$ is equivalent to $\frac{1}{1+t}\tilde{\omega}_{SF}+f^{*}\chi$. Note that $\lambda_{t}^{*}p^{*}f^{*}\chi=p^{*}f^{*}\chi$, so $\lambda_{t}^{*}p^{*}f^{*}\omega_t$ and hence (by \eqref{C2c}) $\lambda_{t}^{*}p^{*}\omega$ is equivalent to $p^{*}(\tilde{\omega}_{SF}+f^{*}\chi)$ on $B\times\mathbb{C}$, i.e., there exists a positive constant $C$ such that
\begin{equation}\label{C2d}
C^{-1}p^{*}(\tilde{\omega}_{SF}+f^{*}\chi)\leq\lambda_{t}^{*}p^{*}\omega(t)\leq Cp^{*}(\tilde{\omega}_{SF}+f^{*}\chi)
\end{equation}
on $B\times\mathbb{C}$. Notice that $p^{*}(\tilde{\omega}_{SF}+f^{*}\chi)$ is $C^{\infty}$ equivalent to $\omega_E$, where $\omega_{E}$ is the Euclidean metric on $B\times\mathbb{C}$. So for each given $K\subset B\times\mathbb{C}$ there exists a positive constant $C_{K}$ such that
\begin{equation}\label{C2e}
C_{K}^{-1}\omega_{E}\leq\lambda_{t}^{*}p^{*}\omega(t)\leq C_{K}\omega_{E}
\end{equation}
on $K$. Combining with \eqref{C2c} we have
\begin{equation}\label{laplace}
-C_{K}\omega_{E}\leq\sqrt{-1}\partial\bar{\partial}(\varphi\circ p\circ\lambda_t)\leq C_{K}\omega_{E}.
\end{equation}

On the other hand, if we fix a $\beta\in C^{\infty}(B,\mathbb{R})$ such that $\chi=\sqrt{-1}\partial\bar{\partial}\beta$ on $B$, then on $B\times\mathbb{C}$ we have
\begin{align}
\lambda_{t}^{*}p^{*}\omega&=\lambda_{t}^{*}p^{*}(\frac{1}{1+t}\omega_{0}+\frac{t}{1+t}f^{*}\chi+\sqrt{-1}\partial\bar{\partial}\varphi)\nonumber\\
&=\frac{1}{1+t}\lambda_{t}^{*}p^{*}\omega_{0}+\frac{t}{1+t}p^{*}f^{*}\chi+\sqrt{-1}\partial\bar{\partial}(\varphi\circ p\circ\lambda_{t})\nonumber\\
&=\frac{1}{1+t}\lambda_{t}^{*}p^{*}(\tilde{\omega}_{SF}-\sqrt{-1}\partial\bar{\partial}\rho)+p^{*}f^{*}\chi+\sqrt{-1}\partial\bar{\partial}(\varphi\circ p\circ\lambda_{t}-\frac{1}{1+t}\beta\circ p\circ f)\nonumber\\
&=p^{*}(\tilde{\omega}_{SF}+f^{*}\chi)+\sqrt{-1}\partial\bar{\partial}(\varphi\circ p\circ\lambda_{t}-\frac{1}{1+t}\beta\circ p\circ f-\frac{1}{1+t}\rho\circ p\circ\lambda_{t})\nonumber.
\end{align}
Set $v=\varphi\circ p\circ\lambda_{t}-\frac{1}{1+t}\beta\circ f\circ p-\frac{1}{1+t}\rho\circ p\circ\lambda_{t}$ on $B\times\mathbb{C}$, which is uniformly bounded on $B\times\mathbb{C}$.
Now we translate (\ref{lt1}) to $B\times\mathbb{C}$ as
\begin{equation}
(\lambda_{t}^{*}p^{*}\omega)^{2}=e^{\frac{t}{1+t}\varphi\circ p\circ\lambda_{t}}\frac{1}{1+t}\lambda_{t}^{*}p^{*}\Omega=e^{\frac{t}{1+t}\varphi\circ p\circ\lambda_{t}}p^{*}\Omega\nonumber,
\end{equation}
where we have used the fact that $\Omega$ only depends on $s\in B$ since $\sqrt{-1}\partial\bar{\partial}\log\Omega=f^{*}\chi$. Hence we arrive at
\begin{equation}\label{2.12}
\log(p^{*}(\tilde{\omega}_{SF}+f^{*}\chi)+\sqrt{-1}\partial\bar{\partial}v)^{2}=\frac{t}{1+t}\varphi\circ p\circ\lambda_{t}+\log p^{*}\Omega.
\end{equation}

\begin{lem}\label{Ck.}
Given any small $K\subset\subset B\times\mathbb{C}$, there exist two constants $C_{K}>0$ and $\alpha_K\in(0,1)$ such that for all $t\geq1$,
\begin{equation}\label{Ck}
\|\lambda_{t}^{*}p^{*}\omega\|_{C^{1,\alpha_K}(K,\omega_{E})}\leq C_{K}.
\end{equation}
\end{lem}

\begin{proof}
Firstly, we may assume that there exists a compact $L\subset\subset B\times\mathbb{C}$ containing $K$ in its interior and some $K_1\subset\subset X_{reg}$ such that $p\circ\lambda_t(L)\subset K_1$. Then by (\ref{laplace}) we can find a positive constant $C_1=C_1(K_1)$ such that
\begin{equation}
|\Delta_{\omega_E}\varphi\circ p\circ\lambda_t|\le C_1
\end{equation}
on $L$. Combining the fact that $\varphi\circ p\circ\lambda_t$ is uniformly bounded, we obtain by elliptic estimates (see \cite{GT}) that, for any fixed $\alpha\in(0,1)$, we can find a slightly smaller subset $L_1$ of $L$ (still contains $K$ in its interior) and a constant $C_2$ depends on $K_1$ and the uniform bound of $\varphi$ such that
\begin{equation}\label{a.0}
\|\varphi\circ p\circ\lambda_t\|_{C^{1,\alpha}(L_1,\omega_E)}\le C_2,
\end{equation}
Similarly, as we have
\begin{equation}
\Delta_{\omega_E}v=tr_{\omega_E}\lambda_t^*p^*\omega(t)-tr_{\omega_E}(p^*(\tilde{\omega_{SF}}+f^*\chi))
\end{equation}
and hence $\Delta_{\omega_E}v$ is bounded in $L$, there exists a slightly smaller subset $L_2$ (still contains $K$ in its interior) of $L_1$ and a positive constant $C_3$ such that
\begin{equation}
\|v\|_{C^{1,\alpha}(L_2,\omega_E)}\le C_3.
\end{equation}
Now we can apply a complex version of Evans-Krylov theory (see e.g. \cite[Theorem 3.1]{B00}) to (\ref{2.12}) to conclude that
for some $\alpha_0\in(0,1)$ and a slightly smaller subset $L_3$ (still contains $K$ in its interior) of $L_2$,
\begin{equation}
\|v\|_{C^{2,\alpha_0}(L_3,\omega_E)}\le C_4.
\end{equation}
Equivalently,
\begin{equation}\label{a.1}
\|\lambda_t^*p^*\omega\|_{C^{\alpha_0}(L_3,\omega_E)}\le C_4,
\end{equation}
and hence
\begin{equation}\label{a.1.5}
\|(\lambda_t^*p^*\omega)^{-1}\|_{C^{\alpha_0}(L_3,\omega_E)}\le C_5.
\end{equation}
Furthermore, for $i\in\{1,2\}$, differentiating (\ref{2.12}) by $\partial_i$ gives
\begin{equation}\label{a.1.75}
\Delta_{\lambda_t^*p^*\omega}(\partial_iv)=\frac{t}{1+t}\partial_i(\varphi\circ p\circ\lambda_t)+A,
\end{equation}
where $A$ is a term whose $C^{\alpha_0}$-norm with respect to $\omega_E$ is uniformly bounded. Then combining (\ref{a.0}), (\ref{a.1}) and (\ref{a.1.5}), we know that the coefficients and right hand side of (\ref{a.1.75}) are in $C^{\alpha_0}(L_2,\omega_E)$ and we can apply Schauder estimates (see \cite{GT}) to conclude that for some $L_4\subset\subset L_3$ (still contains $K$ in its interior),
\begin{equation}
\|\partial_iv\|_{C^{2,\alpha_0}(L_4,\omega_E)}\le C_6\nonumber,
\end{equation}
which implies
\begin{equation}
\|v\|_{C^{3,\alpha_0}(L_4,\omega_E)}\le C_7\nonumber
\end{equation}
and
\begin{equation}
\|\lambda_t^*p^*\omega\|_{C^{1,\alpha_0}(L_4,\omega_E)}\le C_8\nonumber.
\end{equation}
This lemma is proved.
\end{proof}

\begin{rem}
(1) Note that, as mentioned during the above proof, the H\"{o}lder exponent in Lemma \ref{Ck.} is obtained by applying Evans-Krylov theory and hence depends on the chosen compact subset $K$.\\
(2) One may like to apply bootstrapping to derive $C^k$ estimates for all $k$. However, after having $C^{3,\alpha_0}$ bound on $v$, one may not obtain the $C^{3,\alpha}$ bound on $\varphi\circ p\circ\lambda_{t}$ and hence can't apply bootstrapping directly.
\end{rem}

Consequently, as in \cite{GTZ} (see Lamma 4.5), we have

\begin{prop}
Given any $K'\subset\subset U$, there exist two positive constants $C_{K'}$ and $\alpha_{K'}\in(0,1)$ such that for all $t\ge1$,
\begin{equation}\label{Ck'}
\|\omega\|_{C^{1,\alpha_{K'}}(K',\omega_{0})}\leq C_{K'}.
\end{equation}
\end{prop}
Combining with Lemma \ref{conv}, we have
\begin{thm}\label{Cinfty}
For any given compact subset $V$ of $X_{reg}$, there exists a constant $\alpha_{V}\in(0,1)$ such that $\omega(t)\rightarrow f^{*}\chi_{\infty}$ in $C^{1,\alpha_{V}}(V,\omega_0)$-topology as $t\rightarrow\infty$.
\end{thm}
Note that $\chi_{\infty}$ is exactly the unique generalized K\"ahler-Einstein metric on $\Sigma$ (see \cite{ST06}). Hence we have proved part (1) of Theorem \ref{main}.
\par In the remaining part of this section, we shall give a proof of part (2) of Theorem \ref{main}, using the method developed in \cite{TWY}. To this end, we firstly apply a translation to the continuity equation (\ref{lt}) in the following manner. Let $\omega(t)$, $t\in[0,\infty)$, be the unique solution of (\ref{lt}). Set $u=\log(1+t)$ and $\eta(u)=\omega(e^u-1)$ for $t\in[0,\infty)$ Then $\eta=\eta(u)$, $u\in[0,\infty)$, satisfies
\begin{equation}\label{lt.3}
\left\{
\begin{aligned}
e^u\eta(u)&=\omega_{0}-(e^u-1)Ric(\eta(u))\\
\eta(0)&=\omega_{0}.
\end{aligned}
\right.
\end{equation}
and if we set $\eta_u=e^{-u}\omega_0+(1-e^{-u})\chi$ and $\eta(u)=\eta_t+\sqrt{-1}\partial\bar{\partial}\psi(u)$, then (\ref{lt.3}) can be reduced to the following equation of $\psi(u)$.
\begin{equation}\label{lt.4}
\left\{
\begin{aligned}
\psi(u)&=(1-e^{-u})\log\frac{e^u(\eta_u+\sqrt{-1}\partial\bar{\partial}\psi(u))^2}{\Omega}\\
\psi(0)&=0.
\end{aligned}
\right.
\end{equation}
We can easily translate the previous estimates to the current setting as follows.
\begin{lem}\label{est.1}
There exist uniform constants $C>1$ and $\lambda>0$ such that for all $u\ge2$ we have
\begin{itemize}
\item[(1)] $\|\psi(u)\|_{C^0(X)}\le C$;
\item[(2)] $tr_{\eta(u)}f^*\chi\le C$;
\item[(3)] $C^{-1}e^{-C\sigma^{-\lambda}}\eta_u\le\eta(u)\le Ce^{C\sigma^{-\lambda}}\eta_u$;
\item[(4)] $\psi(u)\to\varphi_\infty$ in $L^1(X,\omega_0)$- and $C^{1,\alpha}_{loc}(X_{reg},\omega_0)$-topology, for any $\alpha\in(0,1)$, as $u\to\infty$;
\item[(5)] $tr_{\eta(u)}f^*\chi_{\infty}\le C\sigma^{-\lambda}$.
\end{itemize}
\end{lem}
\begin{proof}
Part (5) is concluded from part (2) and Lemma \ref{blowup'} in the next section.
\end{proof}
Next we prove an analogue of \cite[Lemma 4.6]{TWY}.
\begin{lem}\label{lemma.1}
There exists a positive function $H(u)$ with $H(u)\to0$ as $u\to\infty$ such that
\begin{equation}\label{eq2.2.0}
\sup_{X}e^{-C\sigma^{-\lambda}}|\psi(u)+\partial_u\psi(u)-\varphi_\infty|\le H(u).
\end{equation}
\end{lem}

\begin{proof}
We begin with the following inequality:
\begin{equation}\label{eq2.2.0.5}
\sup_{X}e^{-C\sigma^{-\lambda}}|\psi(u)-\varphi_\infty|\le h(u),
\end{equation}
where $h(u)$ is a positive function and will go to zero as $u\to\infty$. As we have Lemma \ref{est.1}, this can be checked by the same argument in the proof of \cite[Lemma 4.3]{TWY}. Thus it suffices to show
\begin{equation}\label{eq2.2.0.75}
\sup_{X}e^{-C\sigma^{-\lambda}}|\partial_u\psi(u)|\le H(u).
\end{equation}
To this end, we collect some useful equalities as follows. By taking $u$-derivative of (\ref{lt.4}) and using the easy facts that $\Delta_\eta\psi=2-tr_\eta\eta_u$ and $\partial_u\eta_u=\chi-\eta_u$, we have
\begin{equation}\label{eq2.2.1}
(1-e^{-u})\Delta_\eta(\psi+\partial_u\psi)=-(1-e^{-u})(tr_\eta f^*\chi-1)+\partial_u\psi-e^{-u}\log\frac{e^u\eta(u)^2}{\Omega}
\end{equation}
and
\begin{align}\label{eq2.2.2}
&\Delta_\eta\left((1-e^{-u})\partial_u\partial_u\psi+2e^{-u}\partial_u\psi-(1-3e^{-u})\psi\right)\\
&=\partial_u\partial_u\psi+e^{-u}\log\frac{e^u\eta(u)^2}{\Omega}+(1-3e^{-u})tr_{\eta}f^*\chi+4e^{-u}-2+(1-e^{-u})|\partial_u\eta|^2_\eta.
\end{align}
By using parts (1) and (2) of Lemma \ref{est.1} and the maximum principle, one can conclude that there exists a positive constant $C$ such that for all $u\ge2$,
\begin{equation}\label{eq2.2.3}
|\partial_u\psi|_{C^0(X)}\le C
\end{equation}
from (\ref{eq2.2.1}) and
\begin{equation}\label{eq2.2.4}
\partial_u\partial_u\psi\le C
\end{equation}
from (\ref{eq2.2.2}) and (\ref{eq2.2.3}). With (\ref{eq2.2.0.5}), (\ref{eq2.2.3}) and (\ref{eq2.2.4}), we can apply the arguments in Lemma 4.6 of \cite{TWY} to obtain the desired conclusion (\ref{eq2.2.0.75}). The proof is now completed.
\end{proof}

With Lemma \ref{lemma.1} and equation (\ref{eq2.2.1}), one can apply a maximum principle argument (see \cite[Lemma 4.7]{TWY} for details) to conclude that there exist two positive constants $C$ and $\lambda$ such that for all $u\ge2$,
\begin{equation}\label{eq2.2.5}
\sup_Xe^{-C\sigma^{-\lambda}}(tr_\eta f^*\chi_\infty-1)\le C\sqrt{H(u)},
\end{equation}
where $H(u)$ is the function satisfying Lemma \ref{lemma.1}.
\par Now we shall give a proof of part (2) of Theorem \ref{main}, which is equivalent to the following

\begin{prop}\label{fiber conv}
For any $\bar{s}\in\Sigma_{reg}$, $e^u\eta(u)|_{X_{\bar{s}}}$ converges in $C^\infty(X_{\bar{s}},\omega_0|_{X_{\bar{s}}})$-topology to the unique flat metric in class $[\omega_0|_{X_{\bar{s}}}]$ as $u\rightarrow\infty$.
\end{prop}

\begin{proof}
By part (3) of Lemma \ref{est.1}, there exists a constant $C>1$ such that for all $u\ge2$,
\begin{equation}\label{0.1}
C^{-1}\omega_{0}|_{X_{\bar{s}}}\le e^u\eta(u)|_{X_{\bar{s}}}\le C\omega_{0}|_{X_{\bar{s}}}.
\end{equation}
Adapting the arguments in the proof of \cite[Theorem 1.1]{ToZy}, we can find constants $C_k$ for all $k\in\mathbb{N}$ such that for all $u\ge2$,
\begin{equation}\label{Ck1}
\|e^u\eta(u)|_{X_{\bar{s}}}\|_{C^k(X_{\bar{s}},\omega_{0}|_{X_{\bar{s}}})}\leq C_k.
\end{equation}
For the sake of convenience, we sketch a proof here by following \cite{ToZy}. For any given $x\in X_{\bar{s}}$, we fix a small chart $(U,(s,w))$ in $X$ centered at $x$ such that $f$ in this coordinate is given by $f(s,w)=s$. Without loss of any generality, we assume $U=\{(s,w)\in\mathbb{C}^2||s|<1,|w|<1\}$. Let $B_{r}(0)$ be the standard disc in $\mathbb{C}$ centered at $0\in\mathbb{C}$ with radius $r>0$. Define the maps $F_{u}:B_{e^{\frac{u}{2}}}\times B_{1}\to U$ for $u\ge0$ as follows:
\begin{equation}
F_u(s,w)=(se^{-\frac{u}{2}},w)\nonumber.
\end{equation}
Then $e^uF^*_u\eta$ satisfies
\begin{equation}\label{b.1}
Ric(e^uF^*_u\eta)=\frac{1}{e^u-1}F^*_u\omega_0-\frac{1}{e^u-1}(e^uF^*_u\eta)
\end{equation}
on $B_{e^{\frac{u}{2}}}\times B_{1}$ and for any given $M\subset\subset B_{e^{\frac{u}{2}}}\times B_{1}$ one can find a constant $C_{M}>1$ depending only on $M$ such that
\begin{equation}\label{b.2}
C_{M}^{-1}\omega_E\le e^uF^*_u\eta\le C_{M}\omega_E
\end{equation}
on $M$, where $\omega_E$ is the Euclidean metric on $\mathbb{C}^2$. Note that (\ref{b.1}) implies a uniform lower bound for Ricci curvature of $e^uF^*_u\eta$ and both $\frac{1}{e^u-1}F^*_u\omega_0$ and its covariant derivative with respect to $\omega_E$ are uniformly bounded for all $u\ge2$. Thus one can modify Calabi's $C^3$ estimate (see e.g. \cite{PSS,ShWe,HeTo}) to obtain the uniform $C^1$ estimate of $e^uF^*_u\eta$ on a slightly smaller $M_1\subset\subset M$. Moreover, by equation (\ref{b.1}), the components of $e^uF^*_u\eta$ (resp. $F^*_u\omega_0$), say $g_{i\bar{j}}$ (resp. $(g_0)_{i\bar{j}}$), satisfy
\begin{equation}
\Delta_{g}(g_{i\bar{j}})=g^{\bar{q}a}g^{\bar{b}p}\partial_ig_{a\bar{b}}\partial_{\bar{j}}g_{p\bar{q}}-\frac{1}{e^{u}-1}g_{i\bar{j}}+\frac{1}{e^{u}-1}(g_0)_{i\bar{j}}\nonumber.
\end{equation}
where $g$ is the metric associated to $e^uF^*_u\eta$. Thus, a standard bootstrapping argument (see e.g. \cite{HeTo}) will give the higher order estimates for $e^uF^*_u\eta$ and then, by restricting to $\{0\}\times B_1$ and using a finite cover by local charts of $X_{\bar{s}}$, the desired estimate (\ref{Ck1}) follows.
\par On the other hand, for any fixed $\bar{s}\in V'\subset\subset \Sigma_{reg}$ if we define a function $g$ on $f^{-1}(V')\times[2,\infty)$ by
\begin{equation}
g=\frac{e^u\eta|_{X_s}}{\omega_{SF}|_{X_s}},
\end{equation}
then
\begin{equation}
e^u\eta(u)|_{X_s}=g\cdot\omega_{SF}|_{X_s}.
\end{equation}
and
\begin{align}
g&=\frac{e^u\eta|_{X_s}}{\omega_{SF}|_{X_s}}\nonumber\\
&=\frac{e^u\eta\wedge f^*\chi_\infty}{\omega_{SF}\wedge f^*\chi_\infty}\nonumber\\
&=(tr_\eta f^*\chi_\infty)\frac{e^u\eta^2}{2\omega_{SF}\wedge f^*\chi_\infty}\nonumber\\
&=(tr_\eta f^*\chi_\infty)\frac{e^{\frac{e^u}{e^u-1}\psi}}{e^{\varphi_\infty}}\nonumber.
\end{align}
Since \eqref{Ck1} holds uniformly when $s$ varies in any compact subset of $\Sigma_{reg}$, we can make use of the estimates obtained above and \cite[Lemma 2.4]{TWY} to conclude that
\begin{equation}
\sup_{f^{-1}(V')}|g-1|\to0
\end{equation}
and hence
\begin{equation}\label{eq2.2.6}
\|e^u\eta|_{X_{\bar{s}}}-\omega_{SF}|_{X_{\bar{s}}}\|_{C^0(X_{\bar{s}},\omega_0|_{X_{\bar{s}}})}\to0
\end{equation}
as $u\to\infty$. Combining (\ref{Ck1}), we have proved Proposition \ref{fiber conv}, i.e., part (2) of Theorem \ref{main}.
\end{proof}

\begin{rem}
Note that the convergence (\ref{eq2.2.6}) holds uniformly as $s$ varies in any given compact subset of $\Sigma_{reg}$. If we set $\tilde{\eta}_u=e^{-u}\omega_{SF}+(1-e^{-u})f^*\chi_\infty$, then as in \cite{TWY} one easily obtains that for any given $V\subset\subset X_{reg}$,
\begin{equation}
\|\eta(u)-\tilde{\eta}_u\|_{C^0(V,\omega_0)}\to0\nonumber
\end{equation}
and hence
\begin{equation}
\|\eta(u)-f^*\chi_\infty\|_{C^0(V,\omega_0)}\to0\nonumber
\end{equation}
as $u\to\infty$.
\end{rem}

\section{Diameter bounds and Gromov-Hausdorff convergence}\label{geomconv}
In this section we first recall the following lemma, which is proved in \cite[Lemma 3.4]{ST06} and \cite[Section 3.3]{He}.
Recall that $\omega_{SF}$ is the semi-flat $(1,1)$-form on $X_{reg}$ defined by \cite[Lemma 3.1]{ST06} and $F=\frac{\Omega}{2\omega_{SF}\wedge f^{*}\chi}$.

\begin{lem}[\cite{ST06,He}]\label{blowup'}
Let $\Delta_{r}$ be a disk centered at $0$ of small diameter $r$ with respect to $\chi$ such that all fibers $X_{s}$, $s\neq0$, are nonsingular. Let $\Delta_{r}^{*}=\Delta_{r}\setminus\{0\}$. Then there exist two constants $C>1$ and $0<\beta<1$ such that for small $r$,
\begin{equation}\label{blowup}
F(s)\leq\frac{C}{|s|^{2\beta}}
\end{equation}
for all $s\in\Delta_{r}^{*}$.
\end{lem}
\begin{rem}
Combining the results in \cite{ST06} and \cite{He}, $\beta=\max\{\frac{5}{6},1-\frac{1}{2m_{1}},\ldots,1-\frac{1}{2m_{k}}\}$ will satisfy Lemma \ref{blowup'}. Here $m_{i}$'s are the multiplicities of singular fibers as in Section 1.
\end{rem}
Consequently, we have
\begin{prop}\label{diam1.1'}
There exists a constant $C>1$ such that for small $r$
\begin{equation}\label{diam1.1}
diam(\Delta_{r}^{*},\chi_{\infty})\leq Cr^{1-\beta}.
\end{equation}
\end{prop}

\begin{proof}
Without loss of any generality, we assume that $\Delta_{r}$ is the standard disc in $\mathbb{C}$ and $\chi\leq\omega_{E}$ on $\Delta_{r}$, where $\omega_{E}$ is the standard flat metric on $\Delta_{r}$. For any fixed two points $p_{1},p_{2}$ in $\Delta_{r}^{*}$, we express them in polar coordinates by $p_{1}=\rho_{1}e^{\sqrt{-1}\theta_{1}}$, $p_{2}=\rho_{2}e^{\sqrt{-1}\theta_{2}}$, where $\rho_{1}$, $\rho_{1}\in(0,r)$ and $\theta_{1}$, $\theta_{2}\in[0,2\pi)$. Without loss of any generality, we assume $\rho_{1}<\rho_{2}$, $\theta_{1}<\theta_{2}$. Set $p_{3}=\rho_{2}e^{\sqrt{-1}\theta_{1}}$. Then
\begin{equation}
d_{\chi_{\infty}}(p_{1},p_{2})\leq d_{\chi_{\infty}}(p_{1},p_{3})+d_{\chi_{\infty}}(p_{3},p_{2})\nonumber.
\end{equation}
We connect $p_{1}$ and $p_{3}$ by $\gamma_{1}(t)=te^{\sqrt{-1}\theta_{1}}$, $t\in[\rho_{1},\rho_{2}]$ and connect $p_{3}$ and $p_{2}$ by $\gamma_{2}(s)=\rho_{2}e^{\sqrt{-1}s}$, $s\in[\theta_{1},\theta_{2}]$. Then we have
\begin{equation}
L_{\chi_{\infty}}(\gamma_{1})\leq C\int_{\rho_{1}}^{\rho_{2}}\sqrt{F(\gamma_{1}(t))}dt\nonumber
\end{equation}
and
\begin{equation}
L_{\chi_{\infty}}(\gamma_{2})\leq C\rho_{2}\int_{\theta_{1}}^{\theta_{2}}\sqrt{F(\gamma_{2}(s))}ds\nonumber.
\end{equation}
Now by Lemma \ref{blowup'}, we have
\begin{equation}
F(\gamma_{1}(t))\leq C|\gamma_{1}(t)|^{-2\beta}=C|t|^{-2\beta}\nonumber.
\end{equation}
Therefore
\begin{equation}
L_{\chi_{\infty}}(\gamma_{1})\leq C\int_{\rho_{1}}^{\rho_{2}}|t|^{-\beta}dt\leq Cr^{1-\beta}\nonumber.
\end{equation}
On the other hand,
\begin{equation}
L_{\chi_{\infty}}(\gamma_{2})\leq C\rho_{2}^{1-\beta}\int_{\theta_{1}}^{\theta_{2}}ds\leq C r^{1-\beta}\nonumber.
\end{equation}
Thus the (\ref{diam1.1'}) is proved.
\end{proof}

Now it can be seen that the metric completion $(X_{\infty},d_{\infty})$ of $(\Sigma_{reg},\chi_{\infty})$ is compact and $X_{\infty}$ is homeomorphic to $\Sigma$. In fact, for any $s\in\Sigma_{reg}$ and $s_{i}$, define $d_{\infty}(s,s_{i})=\lim_{l\rightarrow\infty}d_{\infty}(s,r_{l}^{i})$, where $\{r_{l}^{i}\}$ is a sequence contained in $\Sigma_{reg}$ and converges to $s_{i}$. Note that Proposition \ref{diam1.1'} implies that this is well-defined. One can define $d_{\infty}(s_{i},s_{j})$ similarly. In particular, we have proved part (3) of Theorem \ref{main}, i.e.,
\begin{prop}\label{completion}
$(X_{\infty},d_{\infty})$ is a compact length metric space and $X_{\infty}$ is homeomorphic to $\Sigma$ as a projective variety.
\end{prop}
We will denote the metric completion of $(\Sigma_{reg},\chi_{\infty})$ by $(\Sigma,d_{\infty})$ and its diameter by $D_{\infty}$.

\par In the following, without loss of any generality we assume $\Sigma\backslash\Sigma_{reg}=\{s_{0}\}$. For small $\delta>0$, let $B_{\infty}(s_{0},\delta)$ be the ball centered at $s_{0}$ of radius $\delta$ with respect to $d_{\infty}$, $K_{\delta}':=\Sigma\backslash B_{\infty}(s_{0},\delta)$ and $K_{\delta}:=f^{-1}(K_\delta')$. Similarly let $H_{\delta}'=\Sigma\backslash B_{\chi}(s_{0},\delta)$ and $H_{\delta}:=f^{-1}(H_\delta')$. We remark that, if $\delta$ is sufficiently small, one can assume that $B_{\chi}(s_{0},\delta)$ is a standard disc in $\mathbb{C}$. Moreover Proposition \ref{diam1.1'} implies that there exists a uniform constant $N>1$ such that for small $\delta$,
\begin{equation}\label{comp}
B_{\chi}(s_{0},\delta)\subset B_{\infty}(s_{0},N\delta^{\frac{1}{N}}).
\end{equation}
\par To begin with, we have
\begin{lem}\label{GH1}
\begin{equation}\label{appr1}
d_{GH}((K_{\delta}',d_{\infty}),(\Sigma,d_{\infty}))\leq\delta.
\end{equation}
\end{lem}

\begin{lem}\label{GH1.5}
For any small $\delta>0$, there exists a constant $T_{\delta}$ such that for all $t\geq T_{\delta}$,
\begin{equation}\label{diam0}
diam(K_{\delta},d_{\omega(t)})\leq D_{\infty}+1.
\end{equation}
\end{lem}

\begin{proof}
For any fixed $p,q\in K_{\delta}$, one can choose a piecewise smooth curve $\gamma(z)\subset H_{(\frac{\delta}{10N})^{N}}'$, $z\in[0,1]$, connecting $f(p)$ and $f(q)$, such that
\begin{equation}\label{0.5}
L_{\chi_{\infty}}(\gamma)\leq d_{\infty}(f(p),f(q))+\frac{\delta}{4}.
\end{equation}
This can be chosen as follows: first choose a piecewise smooth curve $\bar{\gamma}\subset\Sigma_{reg}$ connecting $f(p),f(q)$ with $L_{\chi_{\infty}}(\bar{\gamma})\leq d_{\infty}(f(p),f(q))+\frac{\delta}{10}$. If $\bar{\gamma}\bigcap B_{\chi}(s_{0},(\frac{\delta}{10N})^{N})=\emptyset$, we are done; otherwise we replace the part of $\bar{\gamma}$ contained in $B_{\chi}(s_{0},(\frac{\delta}{10N})^{N})$ by a curve lies in $\partial B_{\chi}(s_{0},(\frac{\delta}{10N})^{N})$ with length with respect to $\chi_{\infty}$ no more than $\frac{\delta}{10}$ (here we have identified $B_{\chi}(s_{0},(\frac{\delta}{10N})^{N})$ with some small standard disc in $\mathbb{C}$). We obtain a curve $\gamma$ as desired.
\par We will lift $\gamma$ to a curve in $H_{(\frac{\delta}{10N})^{N}}$ connecting $p$ and $q$.
\par First, without loss of any generality, we assume that $\gamma(z)$, $z\in[0,1]$, is smooth and covered by two open subsets $U$ and $V$ of $\Sigma_{reg}$ such that $f^{-1}(U)=U\times E$, $f^{-1}(V)=V\times E$, where $E$ is a smooth fiber and both equalities mean diffeomorphisms. Fix a point $r':=\gamma(z_{1})\in U\cap V$. Define $\gamma_{1}(z)=(\gamma(z),e_{1})$ for some $e_{1}\in E$ with $p=(\gamma(0),e_{1})$, $z\in[0,z_{1}]$ and $\gamma_{2}(z)=(\gamma(z-1),e_{2})$ for some $e_{2}\in E$ with $q=(\gamma(1),e_{2})$, $z\in[z_{1}+1,2]$. Also connect $\gamma_{1}(z_{1})$ and $\gamma_{2}(z_{1}+1)$ by a curve $\gamma_{3}(z)\subset E$, $z\in[z_{1},z_{1}+1]$. Now by connecting $\gamma_{1}$, $\gamma_{3}$ and $\gamma_{2}$ we obtain a curve $\sigma(z)\subset H_{(\frac{\delta}{10N})^{N}}$, $z\in[0,2]$, connecting $p$ and $q$. Note that the diameter of smooth fibers over $H_{(\frac{\delta}{10N})^{N}}'$ will go to zero uniformly as $t\rightarrow\infty$, so we can choose $\gamma_{3}$ with arbitrarily small length (with respect to $\omega(t)$) as long as $t$ is large enough. Then we can find a $T_{\delta}$ such that for $t\geq T_{\delta}$,
\begin{align}\label{0.6}
d_{\omega(t)}(p,q)&\leq L_{\omega(t)}(\sigma)\nonumber\\
&\leq L_{\omega(t)}(\gamma_{1})+L_{\omega(t)}(\gamma_{2})+\frac{\delta}{4}\nonumber\\
&\leq L_{f^{*}\chi_{\infty}}(\gamma_{1})+L_{f^{*}\chi_{\infty}}(\gamma_{2})+\frac{\delta}{2}\nonumber\\
&=L_{\chi_{\infty}}(\gamma)+\frac{\delta}{2}\nonumber\\
&\leq d_{\infty}(f(p),f(q))+\frac{3\delta}{4},
\end{align}
where in the third inequality we have used Theorem \ref{Cinfty}.
Thus we obtain
\begin{equation}\label{0.7}
diam(K_{\delta},d_{\omega(t)})\leq D_{\infty}+\delta\nonumber.
\end{equation}
We complete the proof by choosing small $\delta<1$.
\end{proof}

\begin{lem}\label{diam2}
There exist positive constants $D_{1}$ and $T_{1}$ such that for all $t\geq T_{1}$,
\begin{equation}\label{0.12}
diam(X,\omega(t))\leq D_{1}.
\end{equation}
\end{lem}

\begin{proof}
Let $\epsilon>0$ be arbitrary. By Lemma \ref{GH1.5}, we can choose a $K_{\delta}$ and $T_{1}$ such that for all $t\geq T_{1}$,
\begin{equation}\label{0.13}
diam(K_{\delta}, d_{\omega(t)})\leq D_{\infty}+1.
\end{equation}
Using the volume estimate (\ref{vol}) along the continuity method and the fact that $X\setminus X_{reg}$ has real codimension 2 (it is a proper analytic subvariety of $X$), by decreasing $\delta$ if necessary, we have
\begin{equation}\label{0.14}
Vol_{\omega(t)}(X\backslash K_{\delta})\leq(1+t)^{-1}\epsilon
\end{equation}
for all $t\geq T_{1}$.
\par Let $x_{t}\in X\backslash K_{\delta}$ be a point which achieves the maximal distance $R_{t}$ to $K_{\delta}$ in $(X,\omega(t))$. Then $B_{\omega(t)}(x_{t},R_{t})\subset X\backslash K_{\delta}$. On the one hand we have
\begin{align}\label{0.15}
\frac{Vol(B_{\omega(t)}(x_{t},R_{t}+D_\infty+1))}{Vol(B_{\omega(t)}(x_{t},R_{t}))}&\geq
\frac{Vol_{\omega(t)}(X)-Vol_{\omega(t)}(X\backslash K_{\delta})}{Vol_{\omega(t)}(B_{\omega(t)}(x_{t},R_{t}))}\nonumber\\
&\geq\frac{(1+t)^{-1}(C_{1}-\epsilon)}{(1+t)^{-1}\epsilon}\nonumber\\
&\geq C_{2}\epsilon^{-1}.
\end{align}
On the other hand, by the uniform lower bound of Ricci curvature (\ref{Ric bound}), we have
\begin{align}\label{0.16}
\frac{Vol(B_{\omega(t)}(x_{t},R_{t}+D_\infty+1))}{Vol(B_{\omega(t)}(x_{t},R_{t}))}\leq\frac{\int_{0}^{R_{t}+D_\infty+1}sinh^{3}vdv}{\int_{0}^{R_{t}}sinh^{3}vdv}.
\end{align}
Thus if we choose $\epsilon$ small enough, $R_{t}$ will be uniformly bounded and the Lemma \ref{diam2} is proved.
\end{proof}

\begin{lem}\label{GH2}
For any small $\epsilon>0$, there exists a $K_{\delta}$ and a positive constant $T_{2}$ such that for all $t\geq T_{2}$
\begin{equation}\label{appr2}
d_{GH}((K_{\delta},d_{\omega(t)}),(X,d_{\omega(t)}))\leq\epsilon.
\end{equation}
\end{lem}

\begin{proof}
The argument in the proof of Lemma \ref{diam2} implies that, for any $\epsilon>0$, there exist a $K_{\delta}$, two positive constants $C_{2}$ and $T_{2}$ such that for all $t\geq T_{2}$,
\begin{equation}\label{0.17}
C_{2}\epsilon^{-5}\leq\frac{\int_{0}^{R_{t}+D_{\infty}+1}sinh^{3}vdv}{\int_{0}^{R_{t}}sinh^{3}vdv}.
\end{equation}
Since we have known that $R_{t}$ is uniformly bounded, (\ref{0.17}) gives
\begin{equation}\label{0.18}
\int_{0}^{R_{t}}sinh^{3}vdv\leq C_{3}\epsilon^{5}\nonumber,
\end{equation}
which implies that, if $\epsilon$ small enough, for all $t\geq T_{2}$ we have
\begin{equation}\label{0.19}
R_{t}\leq\epsilon\nonumber.
\end{equation}
Thus we conclude (\ref{appr2}).
\end{proof}

\begin{lem}\label{GH3}
For any fixed $\delta>0$, there exists a $\delta_0$ and a $T_{3}>0$ such that for any $p,q\in K_{\delta}$ and $t\geq T_{3}$, one can find a curve $\gamma_{t}\subset K_{\delta_0}$ connecting $p$ and $q$ such that
\begin{equation}\label{GH3.1}
L_{\omega(t)}(\gamma_t)\le d_{\omega(t)}(p,q)+\frac{\delta}{4}.
\end{equation}
\end{lem}

\begin{proof}
This can be proved by an argument similar to that in the proof of Lemma \ref{GH1.5}. In fact, if the minimal geodesic $\gamma_t$ in $(X,\omega(t))$ connecting $p$ and $q$ intersects $f^{-1}(B_{\chi}(s_0,\left(\frac{\delta}{10N}\right)^{N}))$, then we may replace the part of $\gamma_t$ contained in $f^{-1}(B_{\chi}(s_0,\left(\frac{\delta}{10N}\right)^{N}))$ by lifting suitably a circle $\gamma'\subset\partial B_{\chi}(s_0,\left(\frac{\delta}{10N}\right)^{N})$ and obtain a new curve, still denote it by $\gamma_t$, satisfying \eqref{GH3.1}. Now we complete the proof by choosing a sufficiently large $T_3$ and a $\delta_0<\delta$ with $H_{\left(\frac{\delta}{10N}\right)^{N}}\subset K_{\delta_0} $.
\end{proof}

\begin{lem}\label{GH4}
There exists a $T_{\delta}>0$ such that for all $t\geq T_{\delta}$,
\begin{equation}\label{0.26}
d_{GH}((K_{\delta},d_{\omega(t)}),(K_{\delta}',d_{\infty}))\leq\delta.
\end{equation}
\end{lem}

\begin{proof}
Define a map $g:K_{\delta}'\rightarrow K_{\delta}$ by mapping every point $s\in K_{\delta}'$ to some chosen point $g(s)$ in $X_{s}$. Given arbitrary $p,q\in K_{\delta}$ and $s,r\in K_{\delta}'$. By Lemma \ref{GH3}, for $t$ large enough, we can find a curve $\gamma_{t}\subset K_{\delta_0}$ connecting $p$ and $q$ such that
\begin{equation}\label{0.27}
L_{\omega(t)}(\gamma_{t})\leq d_{\omega(t)}(p,q)+\frac{\delta}{4}.
\end{equation}
By Theorem \ref{Cinfty}, we can find a $T_{\delta}$ such that for all $t\geq T_{\delta}$,
\begin{align}\label{0.28}
L_{\omega(t)}(\gamma_{t})&\geq L_{f^{*}\chi_{\infty}}(\gamma_{t})-\frac{\delta}{4}\nonumber\\
&\geq L_{\chi_{\infty}}(f\circ\gamma_{t})-\frac{\delta}{4}\nonumber\\
&\geq d_{\infty}(f(p),f(q))-\frac{\delta}{4}\nonumber.
\end{align}
Therefore
\begin{equation}\label{0.29}
d_{\infty}(f(p),f(q))\leq d_{\omega(t)}(p,q)+\delta.
\end{equation}
A similar argument gives
\begin{equation}\label{0.29.5}
d_{\infty}(s,r)\leq d_{\omega(t)}(g(s),g(r))+\delta.
\end{equation}
On the other hand, using the argument in Lemma \ref{GH1.5}, we can show that, for all $t\geq T_{\delta}$ (increase $T_{\delta}$ if necessary),
\begin{equation}\label{0.32}
d_{\omega(t)}(g(s),g(r))\leq d_{\infty}(s,r)+\delta,
\end{equation}
and
\begin{equation}\label{0.33}
d_{\omega(t)}(p,q)\leq d_{\infty}(f(p),f(q))+\delta.
\end{equation}
Also of course we have, for any $s\in K_{\delta}'$,
\begin{equation}\label{0.34}
d_{\infty}(s,f\circ g(s))\equiv0
\end{equation}
and for any $p\in K_{\delta}$,
\begin{equation}\label{0.35}
d_{\omega(t)}(p,g\circ f(p))\leq\delta
\end{equation}
for all $t\geq T_{\delta}$.
Combining (\ref{0.29}), (\ref{0.29.5}), (\ref{0.32}), (\ref{0.33}), (\ref{0.34}) and (\ref{0.35}), we conclude Lemma \ref{GH4}.
\end{proof}

Now we are ready to prove part (4) of Theorem \ref{main}, i.e.,
\begin{thm}\label{GH5}
$(X,\omega(t))\rightarrow(\Sigma,d_{\infty})$ in Gromov-Hausdorff topology as $t\rightarrow\infty$.
\end{thm}

\begin{proof}
For any small $\epsilon>0$, we fix a $\delta>0$ satisfying
\begin{itemize}
\item[(1)] $\delta<\frac{\epsilon}{4}$;
\item[(2)] $d_{GH}((K_{\delta}',d_{\infty}),(\Sigma,d_{\infty}))<\frac{\epsilon}{4}$;
\item[(3)] $d_{GH}((K_{\delta},d_{\omega(t)}),(X,d_{\omega(t)}))<\frac{\epsilon}{4}$ for any $t\geq T_{\epsilon}$;
\item[(4)] $d_{GH}((K_{\delta},d_{\omega(t)}),(K_{\delta}',d_{\infty}))<\frac{\epsilon}{4}$ for any $t\geq T_{\epsilon}$.
\end{itemize}
Note that (2), (3) and (4) are guaranteed by Lemma \ref{GH1}, Lemma \ref{GH2} and Lemma \ref{GH4} respectively. In fact the constant $T_{\epsilon}$ may also depend on $\delta$. However $\delta$ is a fixed constant determined by $\epsilon$. Therefore $T_{\epsilon}$ only depends on $\epsilon$ and Theorem \ref{GH5} is proved.
\end{proof}

We end this paper by a remark on the K\"{a}hler-Ricci flow.
\begin{rem}\label{r3.1}
Consider the K\"{a}hler-Ricci flow $\bar{\omega}(t)$ on $X$ starting from $\omega_{0}$,
\begin{equation}
\left\{
\begin{aligned}
\partial_{t}\bar{\omega}(t)&=-Ric(\bar{\omega}(t))-\bar{\omega}(t)\\
\bar{\omega}(0)&=\omega_{0}\nonumber,
\end{aligned}
\right.
\end{equation}
which exists for all $t\geq0$ \cite{C,TZo,Ts} and converges to $f^{*}\chi_{\infty}$ in $C_{loc}^\infty(X_{reg},\omega_0)$-topology \cite{FZ} (see also \cite{To15}). If $f:X\rightarrow\Sigma$ has only singular fibers of type $mI_{0}$, i.e., every singular fiber $X_{s_i}$ is a smooth elliptic curve of some non-trivial multiplicity $m_i$ (see, e.g., \cite{M} for more discussions on all possible singular fibers), it is shown in \cite{ToZy} that $|Rm(\bar{\omega}(t))|_{\bar{\omega}(t)}$ is uniformly bounded on $X\times[0,\infty)$. In particular, $Ric(\bar{\omega}(t))$ is uniformly bounded from below on $X\times[0,\infty)$. Thus in this case we can apply the same arguments to show that, as $t\rightarrow\infty$, $(X,\bar{\omega}(t))\rightarrow(\Sigma,d_{\infty})$ in Gromov-Hausdorff topology\footnote{V. Tosatti points out to us that if the only singular fibers are of type $mI_0$, then $X$ has a \emph{global} finite covering space which is an elliptic bundle and hence one can also use the similar arguments in \cite{TWY15} to conclude this result.}. Note that if $X$ has no singular fiber, the Gromov-Hausdorff convergence is known, see \cite{FZ,TWY}.
\end{rem}

\section*{Acknowledgements}
Y. Zhang is grateful to Prof. Huai-Dong Cao for constant encouragement and support. He is also grateful to Prof. Valentino Tosatti and Dr. Zhiming Lin for very useful discussions, in particular to Prof. Tosatti for answering his questions on papers \cite{He,To10}. Part of this work was carried out while Y. Zhang was visiting Capital Normal University in Beijing, which he would like to thank for providing warm hospitality and nice working environment.
\par Both authors thank Yanir Rubinstein for pointing out that the continuity method \eqref{lt} was also introduced in his paper \cite{Ru}, Valentino Tosatti for giving a comment on Remark \ref{r3.1} and the referee for useful comments.

\end{document}